\documentclass[a4paper]{amsart}
\usepackage{wrapfig}
\usepackage[dvips]{graphicx}
\usepackage{amsmath,amsthm,amssymb,amscd}
\theoremstyle{definition}
\newtheorem{thm}{Theorem}[section]
\newtheorem{Def}[thm]{Definition}
\newtheorem{pro}[thm]{Proposition}
\newtheorem{cor}[thm]{Corollary}
\newtheorem{lem}[thm]{Lemma}
\newtheorem{ex}[thm]{Example}
\newtheorem{rem}[thm]{Remark}
\theoremstyle{definition}

\begin{document}

\title{Fundamental group of uniquely ergodic Cantor minimal systems}
\author{Norio Nawata}
\address[Norio Nawata]{Institute of Mathematics for Industry, 
Kyushu University, Motooka, 
Fukuoka, 819-0395,  Japan}      
\email{n-nawata@math.kyushu-u.ac.jp}
\begin{abstract}
We introduce the fundamental group ${\mathcal F}(\mathcal{R}_{G, \varphi})$ of a uniquely ergodic 
Cantor minimal $G$-system $\mathcal{R}_{G, \varphi}$ where $G$ is a countable discrete group. 
We compute fundamental groups of several uniquely ergodic Cantor minimal $G$-systems. 
We show that if $\mathcal{R}_{G, \varphi}$ arises from a free action $\varphi$ of a finitely 
generated abelian group, then there exists a unital countable subring $R$ of $\mathbb{R}$ such 
that $\mathcal{F}(\mathcal{R}_{G, \varphi})=R_{+}^\times$. 
We also consider the relation between fundamental groups of uniquely ergodic 
Cantor minimal $\mathbb{Z}^n$-systems and fundamental groups of crossed product $C^*$-algebras 
$C(X)\rtimes_{\varphi} \mathbb{Z}^n$. 
\end{abstract}
\maketitle

\section{Introduction} 
Let $M$ be a factor of type $\mathrm{II}_1$ with a normalized trace $\tau$. 
Murray and von Neumann introduced the fundamental group ${\mathcal F}(M)$ of $M$ in \cite{MN}. 
The fundamental group of $M$ is defined as the set of the numbers $\tau \otimes Tr(p)$ 
for some projection $p \in M_n(M)$ such that $pM_n(M)p$ is isomorphic to $M$ where 
$Tr$ is the usual unnormalized trace on $M_n(\mathbb{C})$. 
They showed that if $M$ is  hyperfinite, then ${\mathcal F}(M) = {\mathbb R_+^{\times}}$. 
Since then there has been many works on the computation of fundamental groups. 
Voiculescu \cite{Vo} showed that ${\mathcal F}(L(\mathbb{F}_{\infty}))$ of the group factor 
of the free group $\mathbb{F}_{\infty}$ contains the positive rationals and 
Radulescu proved that 
${\mathcal F}(L(\mathbb{F}_{\infty})) = {\mathbb R}_+^{\times}$ in 
\cite{Ra}. Connes \cite{Co} showed that if $G$ is an ICC group with property 
(T), then  ${\mathcal F}(L(G))$ is a countable group. Popa 
showed that any countable subgroup of $\mathbb R_+^{\times}$ 
can be realized as the fundamental group of some factor of type $\mathrm{II}_1$ in \cite{Po1}. 
Furthermore Popa and Vaes \cite{PV} exhibited a large family $\mathcal{S}$ 
of subgroups of $\mathbb{R}_{+}^\times$, containing $\mathbb{R}_{+}^\times$ 
itself, all of its countable subgroups, as well as uncountable subgroups with 
any Hausdorff dimension in $(0,1)$, such that for each $H\in\mathcal{S}$ 
there exist many free ergodic measure preserving actions of $\mathbb{F}_{\infty}$ 
for which the associated $\mathrm{II}_1$ factor $M$ has the fundamental group equal to $H$. 

Popa's results are based on the study of rigidity properties of II$_1$ factors 
$L^{\infty} (X)\rtimes_{T} G$ and orbit equivalence relations $\mathcal{R}_{G,T}$ arising from 
(free) ergodic measure preserving actions $T$ of countable groups $G$ on probability measure 
spaces $(X, \mu)$ via the group measure space construction in \cite{MN1}. 
(See also \cite{FM1} and \cite{FM2} for its generalization.) 
The fundamental group of $\mathcal{R}_{G, T}$ can also be defined as the set of numbers 
$\mu \times \delta (Y)$ for some measurable set $Y$ in $X\times \{1,..,n\}$ such that 
$\mathcal{R}_{G, T}^{n}|_{Y}$ is orbit equivalent to $\mathcal{R}_{G, T}$ where 
$\delta$ is the counting measure. It is easy to see that $\mathcal{F}(\mathcal{R}_{G, T})\subseteq 
\mathcal{F}(L^{\infty}(X) \rtimes_{T} G)$. This inclusion may be strict. (See 6.1 in \cite{Pspectral}.) 

The result in \cite{CFW} implies $\mathcal{F}(\mathcal{R}_{G, T})=\mathbb{R}_{+}^\times$ 
whenever $G$ is amenable. 
Gefter and Golodets \cite{GG} showed that there exist ergodic measure preserving actions having 
trivial fundamental group before Popa's result \cite{PoAnn}. 
Moreover Gaboriau \cite{G} showed that if $\mathcal{R}_{\mathbb{F}_{n}, T}$ is an orbit equivalence 
relation arising from a free ergodic measure preserving action $T$ of a non-amenable free group of 
finite rank, then $\mathcal{F}(\mathcal{R}_{\mathbb{F}_{n}, T})=\{1 \}$. 
On the contrary, Popa and Vaes showed that for each $H\in\mathcal{S}$ 
there exist many orbit equivalence relations $\mathcal{R}_{\mathbb{F}_{\infty}, T}$ arising from 
free ergodic measure preserving actions $T$ of $\mathbb{F}_{\infty}$ such that 
$\mathcal{F}(\mathcal{R}_{\mathbb{F}_{\infty}, T})=H$. 
For a countable discrete group $G$, it is an interesting problem to consider how $G$ 
affect $\mathcal{F}(\mathcal{R}_{G,T})$. 

Watatani and the author introduced the fundamental group $\mathcal{F}(A)$ 
of a simple unital $C^*$-algebra $A$ with a normalized trace $\tau$ 
based on the computation  of Picard groups by Kodaka \cite{kod1}, \cite{kod2} and \cite{kod3}. 
The fundamental group ${\mathcal F}(A)$ is 
defined as the set of 
the numbers $\tau \otimes Tr(p)$ for some projection 
$p \in M_n(A)$ such that $pM_n(A)p$ is isomorphic to $A$. 
We computed the fundamental groups of several $C^*$-algebras and showed 
that any countable subgroup of $\mathbb{R}_+^\times$ 
can be realized as the fundamental group of a separable simple unital $C^*$-algebra 
with a unique trace in \cite{NW} and \cite{NW2}. 
Note that the fundamental groups of separable simple unital $C^*$-algebras are countable. 

We consider topological dynamical systems on the Cantor set in this paper. 
Giordano, Putnam and Skau classified Cantor minimal $\mathbb{Z}$-systems up to (topological) 
orbit equivalence \cite{GPS1}. 
They also introduced the notion of strong orbit equivalence for Cantor minimal 
$\mathbb{Z}$-system and showed that two Cantor minimal $\mathbb{Z}$-system are strong orbit 
equivalent if and only if their associated $C^*$-algebras are isomorphic. 
Giordano, Matui, Putnam and Skau extends the classification up to orbit equivalence of 
Cantor minimal systems arising from actions of finitely generated abelian groups 
\cite{GMPS1} and \cite{GMPS2}. 

In this paper we introduce the fundamental group of a uniquely ergodic 
Cantor minimal $G$-system $\mathcal{R}_{G, \varphi}$ where $G$ is a countable discrete group. 
The fundamental group $\mathcal{F}(\mathcal{R}_{G, \varphi})$ of a uniquely ergodic Cantor 
minimal $G$-system $\mathcal{R}_{G, \varphi}$ is defined as the set of numbers 
$\mu \times \delta (U)$ for some clopen set $U$ in $X\times \{1,..,n\}$ such that 
$\mathcal{R}_{G, \varphi}^{n}|_{U}$ is (topologically) orbit equivalent to 
$\mathcal{R}_{G, \varphi}$ where $\mu$ is a unique invariant probability measure and $\delta$ is 
the counting measure. We show that $\mathcal{F}(\mathcal{R}_{G, \varphi})$ is a countable 
multiplicative subgroup of $\mathbb{R}_{+}^\times$. 
Note that we cannot show this theorem in the same way as in the case of ergodic measure 
preserving actions of countable groups on probability measure spaces because we do not know 
whether the analogous result of Hopf equivalence theorem (see, for example, Proposition 3.3 
in \cite{FM1}) is true for Cantor minimal $G$-systems. Moreover we do not know whether 
$\mu\times \delta (U_1 )=\mu\times \delta (U_2)$ implies that 
$\mathcal{R}_{G, \varphi}^{n}|_{U_1}$ is orbit equivalent to 
$\mathcal{R}_{G, \varphi}^{n}|_{U_2}$ in general. 

We show that for any unital countable subring $R$ of $\mathbb{R}$, 
there exists a Cantor minimal $\mathbb{Z}$-system $\mathcal{R}_{\mathbb{Z}, \varphi}$ such that 
$\mathcal{F}(\mathcal{R}_{\mathbb{Z}, \varphi})=R_{+}^\times$. 
Conversely, we also show that if $\mathcal{R}_{G, \varphi}$ arises from a free action of a finitely 
generated abelian group, then there exists a unital countable subring $R$ of $\mathbb{R}$ such 
that $\mathcal{F}(\mathcal{R}_{G, \varphi})=R_{+}^\times$. 
Therefore $\{9^n:n\in\mathbb{Z}\}$ cannot be realized as the 
fundamental group of a Cantor minimal system in this class. 

We do not know whether there exists a relation between orbit 
equivalence of Cantor minimal $G$-systems and $C^*$-isomorphism of associated $C^*$-algebras 
in general. 
We show that if $\mathcal{R}_{\mathbb{Z}^n, \varphi}$ arises from a free minimal action $\varphi$ 
of $\mathbb{Z}^n$ on a Cantor set $X$ , then $\mathcal{F}(C(X)\rtimes_{\varphi} \mathbb{Z}^n)
\subseteq \mathcal{F}(\mathcal{R}_{\mathbb{Z}^n, \varphi})$. 
Note that we do not know whether $C(X)\rtimes_{\varphi} \mathbb{Z}^2$ belongs to classifiable 
classes by $K$-groups. 
We also show that there exists a Cantor minimal $\mathbb{Z}$-system $R_{\mathbb{Z}, \varphi}$ 
such that $\mathcal{F}(C(X)\rtimes_{\varphi} \mathbb{Z})\neq \mathcal{F}(\mathcal{R}_{\mathbb{Z}, \varphi})$. 

In section \ref{sec:examples} 
we compute fundamental groups of several uniquely ergodic Cantor minimal systems. 
Some computations suggest that the fundamental groups of Cantor minimal systems have 
arithmetical flavor compared with the case of ergodic measure preserving actions on 
probability measure spaces.

\section{Cantor minimal systems}\label{sec:minimal system}
In this section we review some definitions and results on Cantor minimal systems. 
See, for example, \cite{GMPS2} and \cite{GPS2} for details. 
Let $X$ be the Cantor set, that is, $X$ is a compact separable totally disconnected metric 
space without isolated points. For a free action $\varphi$ of a countable discrete group $G$ on $X$ 
by homeomorphisms (we call this dynamical system the \textit{Cantor $G$-system}), 
an orbit equivalence relation $\mathcal{R}_{G, \varphi}$ is defined by 
$$
\mathcal{R}_{G, \varphi} =\{(x,\varphi_{g}(x))\in X\times X:x\in X, g\in G \} .
$$
For two orbit equivalence relations $\mathcal{R}_1$ on $X_1$ and 
$\mathcal{R}_2$ on $X_2$, we say that 
$\mathcal{R}_1$ is \textit{(topologically) orbit equivalent} to 
$\mathcal{R}_2$ if there exists a homeomorphism $F$ of $X_1$ onto 
$X_2$ such that $F\times F(\mathcal{R}_1)=\mathcal{R}_2$. 
We denote by $\mathrm{Aut}(\mathcal{R})$ the set of 
homeomorphism of $X$ such that $F\times F(\mathcal{R})=\mathcal{R}$. 
An orbit equivalence relation $\mathcal{R}$ on $X$ is said to be \textit{minimal} if 
$\mathcal{R}[x]=\{y\in X:(x,y)\in \mathcal{R}\}$ is dense in $X$ 
for any $x\in X$. 
An \textit{$\mathcal{R}$-invariant measure} is a Borel measure $\mu$ on $X$ such 
that $\mu$ is $h$-invariant for any $h\in \mathrm{Homeo}(X)$ satisfying 
$(x,h(x))\in \mathcal{R}$ for all $x\in X$. 
Let $M(\mathcal{R})$ denote the set of $\mathcal{R}$-invariant probability measures. 
An orbit equivalence relation $\mathcal{R}$ is said to be \textit{uniquely ergodic} if 
$M(\mathcal{R})$ has exactly one element. 
These definitions coincide with usual definitions for topological dynamical systems. 

Assume that $M(\mathcal{R})\neq \phi$. 
Let $D_{m}(X, \mathcal{R})$ be a quotient group of $C(X,\mathbb{Z})$ 
by 
$
\{f\in C(X,\mathbb{Z}):\int_{X}fd\mu =0\;\mathrm{for}\;\mathrm{all}\;\mu \in M
(\mathcal{R}) \}
$. 
This countable abelian group $D_m(\mathcal{R})$ has an order structure 
such that the positive cone $D_m(\mathcal{R})^{+}$ is 
the set of all $[f]$, where $f\geq 0$. For a clopen set $U$ of $X$, we denote by 
$1_{U}$ the characteristic function on $U$, that is, $1_{U} (x)=1$ if 
$x\in U$ and 
$1_{U} (x)=0$ if $x\notin U$. 
We say that two triples $(D_m(\mathcal{R}_1),D_m(\mathcal{R}_1)^{+}, [1]_{X_1})$ and 
$(D_m(\mathcal{R}_2),D_m(\mathcal{R}_2)^{+}, [1]_{X_2})$  
are \textit{isomorphic} if there exists an order isomorphism $\psi$ of 
$D_m(\mathcal{R}_1)$ onto $D_m(\mathcal{R}_2)$ such that 
$\psi ([1_{X_1}]) =[1_{X_2}]$. 
A triple $(D_m(\mathcal{R}),D_m(\mathcal{R})^{+}, [1]_{X})$ is an orbit equivalence invariant 
for Cantor minimal systems. 
Giordano, Putnam, Matui and Skau actually showed that this triple is a complete invariant 
for AF relations and the case where $G$ is a finitely generated abelian group. 
(See Theorem 2.5 in \cite{GMPS2}.) 
\begin{rem}
If $R$ is uniquely ergodic with an $R$-invariant probability measure $\mu$, then the map 
$T_{\mu}$ of $D_m(\mathcal{R})$ into $\mathbb{R}$ defined  by $T_{\mu}([f])=\int_{X}fd\mu$ 
is an order isomorphism of $D_m(\mathcal{R})$ onto $T_{\mu}(D_m(\mathcal{R}))\subseteq \mathbb{R}$ 
with $T_{\mu}([1]_{X})=1$. Moreover $T_{\mu}(D_m(\mathcal{R}))$ is an additive subgroup of 
$\mathbb{R}$ generated by $\{\mu (U):U\text{ is a clopen set in }X \}$ because for any 
$f\in C(X,\mathbb{Z})$, there exist clopen sets $U_1,..,U_n$ in $X$ and 
integers $m_1,..,m_n$ such that $f=\sum_{k=1}^nm_k 1_{U_k}$. 
\end{rem}
For a Cantor system $\mathcal{R}_{G, \varphi}$ and a natural number $n$, 
define an orbit equivalence relation $\mathcal{R}_{G, \varphi}^n$ 
on $X\times \{1,..,n \}$ by $\{((x,i),(y,j)) :(x,y)\in\mathcal{R}_{G, \varphi}, i,j\in \{1,..,n \} \}$. 
Note that $X\times \{1,..,n \}$ is a Cantor set and $\mathcal{R}_{G, \varphi}^n$ arises from 
an action $\tilde{\varphi}$ of $G\times \mathbb{Z}/n\mathbb{Z}$ on $X\times \mathbb{Z}/n\mathbb{Z}$. 
If $\mathcal{R}_{G, \varphi}$ is a Cantor minimal system, then 
$\mathcal{R}_{G, \varphi}^n$ is minimal. 
It can be easily checked that for any $\mathcal{R}_{G, \varphi}^n$-invariant 
probability measure 
$\nu$ on $X\times \{1,..,n\}$, there exists 
an $\mathcal{R}_{G, \varphi}$-invariant probability measure $\mu$ on $X$ such that  
$\nu =\mu\times \frac{1}{n}\delta$ where $\delta$ is the counting measure on $\{1,..,n\}$. 
\begin{pro}
Let $\mathcal{R}_{G, \varphi}$ be a Cantor system and $n$ a natural number. 
Then 
$$(D_m(\mathcal{R}_{G, \varphi}^n), D_m(\mathcal{R}_{G, \varphi}^n)^+, [1_{X\times \{1,..,n \}}]) 
\cong 
(D_m(\mathcal{R}_{G, \varphi}), D_m(\mathcal{R}_{G, \varphi})^{+}, n[1_{X}]).
$$ 
\end{pro}
\begin{proof}
Let $\psi$ be an order homomorphism of $D_m(\mathcal{R}_{G, \varphi})$ to 
$D_m(\mathcal{R}_{G, \varphi}^n)$ such that $\psi ([1_{V}]) =[1_{V\times \{1\}}]$ 
for any clopen set $V$ in $X$. 
It is easy to see that $\psi$ is well-defined and injective. 
For any clopen set $U$ in $X\times\{1,..,n\}$, there exist 
clopen sets $V_1$,..,$V_n$ in $X$ such that $U=\cup_{k=1}^nV_k\times \{k \}$. 
Since $[1_{V_k\times\{k\}}]=[1_{V_k\times\{1\}}]$ in 
$D_m(\mathcal{R}_{G, \varphi}^n)$ for any $1\leq k\leq n$, we see that 
$\psi$ is surjective and $n\psi ([1_{X}]) =[1_{X\times\{1,..,n\}}]$. 
Therefore we obtain the conclusion. 
\end{proof}
Let $U$ be a clopen set in $X$. 
Define an orbit equivalence relation $\mathcal{R}_{G, \varphi}|_{U}$ 
by $\mathcal{R}_{G, \varphi}\cap U\times U$. 
If $\mathcal{R}_{\mathbb{Z}, \varphi}$ is a Cantor minimal $\mathbb{Z}$-system, 
then $\mathcal{R}_{\mathbb{Z}, \varphi}|_{U}$ is equal to the induced system 
$\mathcal{R}_{\mathbb{Z}, \varphi_{U}}$ in \cite{GPS1}. 
It is easy to see that if $\mathcal{R}_{G, \varphi}$ is minimal, then 
$\mathcal{R}_{G, \varphi}|_{U}$ is minimal. 
\begin{pro}\label{pro:inclusion}
Let $\mathcal{R}_{G, \varphi}$ be a Cantor minimal system and $U$ a clopen set in $X$. Then 
there exists a bijective map $\pi$ of $M(\mathcal{R}_{G, \varphi})$ onto  
$M(\mathcal{R}_{G, \varphi}|_{U})$ such that $\pi (\mu)=\mu |_{U}$.  Moreover we have 
$$(D_m(\mathcal{R}_{G, \varphi}|_{U}), D_m(\mathcal{R}_{G, \varphi}|_{U})^+, 
[1_{U}]) 
\cong 
(D_m(\mathcal{R}_{G, \varphi}), D_m(\mathcal{R}_{G, \varphi})^{+}, [1_{U}]).
$$ 
\end{pro}
\begin{proof}
It is easy to see that $\pi$ is well-defined. 
Since $X$ is compact and $\mathcal{R}_{G, \varphi}$ is minimal, 
there exist $g_1,..,g_n\in G$ such that $X=\cup_{k=1}^n\varphi_{g_k}(U)$. 
Put $U_{k}:=\varphi_{g_k}(U)\backslash (\cup_{i=1}^{k-1}\varphi_{g_i}(U))$ 
for $1\leq k\leq n$. 
Then we have 
$\mu (V)=\sum_{k=1}^n\mu (\varphi_{g_k^{-1}}(V\cap U_k))$ for any 
$\mu\in M(\mathcal{R}_{G, \varphi})$ and any clopen set $V$ in $X$. 
Since $\varphi_{g_k^{-1}}(V\cap U_k)$ is a clopen set in $U$, $\mu$ is determined by 
$\mu |_{U}$. Therefore $\pi$ is injective. 
For any $\nu\in M(\mathcal{R}_{G, \varphi}|_U)$, define a measure $\mu$ on $X$ by 
$\mu (V)=\sum_{k=1}^n\nu (\varphi_{g_k^{-1}}(V\cap U_k))$ 
for any clopen set $V$ in $X$. 
We shall show that $\mu$ is $\mathcal{R}_{G, \varphi}$-invariant. 
Note that for $1\leq k,j\leq n$, $g\in G$ and a clopen set $V$ in $X$,  
there exist homeomorphisms $h_{k,j,g}$ of $U$ such that 
$h_{k,j,g}(x)=x$ if $x\notin \varphi_{g_k^{-1}}(U_k\cap\varphi_{g}(U_j\cap V))$ 
and $h_{k,j,g}(x)=\varphi_{g_j}^{-1}\circ\varphi_{g}^{-1}\circ\varphi_{g_k}(x)$ if 
$x\in \varphi_{g_k^{-1}}(U_k\cap\varphi_{g}(U_j\cap V))$. 
It is clear that $(x, h_{k,j,g}(x))\in \mathcal{R}_{G, \varphi}|_U$ for any $x\in U$. 
For any $g\in G$ and any clopen set $V$ in $X$, we have 
\begin{align*}
\mu (\varphi_{g}(V)) 
& =\sum_{k=1}^n \nu (\varphi_{g_k^{-1}}(U_k \cap \varphi_{g} (V))) \\ 
& =\sum_{k=1}^n \nu (\varphi_{g_k^{-1}}(U_k \cap \varphi_{g} (\cup_{j=1}^n U_j\cap V))) \\
& =\sum_{k=1}^n \sum_{j=1}^n \nu (\varphi_{g_k^{-1}}(U_k\cap \varphi_{g} (U_j\cap V))) \\
& = \sum_{k=1}^n \sum_{j=1}^n \nu (h_{k,j,g} (\varphi_{g_k^{-1}}(U_k\cap \varphi_{g} (U_j\cap V)))) \\
& =\sum_{k=1}^n \sum_{j=1}^n \nu (\varphi_{g_j^{-1}}( \varphi_{g^{-1}}(U_k)\cap U_{j}\cap V)) \\
&=\sum_{j=1}^n\nu (\varphi_{g_j^{-1}}(\varphi_{g^{-1}} (\cup_{k=1}^n U_k)\cap U_j\cap V)) \\
& = \mu (V).
\end{align*}
Therefore we see that $\pi$ is surjective. 

Define a homomorphism $\psi$ of $D_m(\mathcal{R}_{G, \varphi})|_{U}$ to 
$D_m(\mathcal{R}_{G, \varphi})$ by $\psi ([1_{W}]) = [1_{W}]$ 
for any clopen set $W$ in $U$. 
Then it can be easily checked that $\psi$ is an order isomorphism by a similar 
argument. 
\end{proof}
\section{Analogous result of Brown's Lemma}
In this section we shall show the analogous result of Lemma 2.5 in \cite{B}. 
Let $\mathcal{R}_{G, \varphi}$ be a Cantor minimal system. 
Define an orbit equivalence relation $\mathcal{R}_{G, \varphi}^\infty$ 
on $X\times \mathbb{N}$ by $\{((x,i),(y,j)) :(x,y)\in\mathcal{R}_{G, \varphi}, i,j\in \mathbb{N} \}$. 
Note that $X\times \mathbb{N}$ is a locally compact separable totally disconnected metric 
space without isolated points and $\mathcal{R}_{G, \varphi}^\infty$ 
arises from an action $\tilde{\varphi}$ of $G\times \mathbb{Z}$ on $X\times \mathbb{Z}$. 
It can be easily checked that if $\mathcal{R}_{G, \varphi}$ is a uniquely ergodic 
Cantor minimal system, then $\mathcal{R}_{G, \varphi}^\infty$ is minimal and 
has unique (up to scalar multiple) $\mathcal{R}_{G, \varphi}^\infty$-invariant measure. 
\begin{lem}\label{lem:lemma of key lemma}
Let $\mathcal{R}_{G, \varphi}$ be a Cantor minimal system on $X$ and 
$U$ a clopen set in $X$. Then there exists 
an injective continuous open map $F$ of $X\times \mathbb{N}$ to 
$U\times \mathbb{N}$ such that $(y_1, y_2)\in\mathcal{R}_{G, \varphi}^{\infty}$ if and only if 
$(F (y_1),F (y_2))\in\mathcal{R}_{G, \varphi}^{\infty}$. 
\end{lem}
\begin{proof}
Since $X$ is compact and $\mathcal{R}_{G, \varphi}$ is minimal, 
there exist $g_1,..,g_n\in G$ such that $X=\cup_{k=1}^n\varphi_{g_k}(U)$. 
Put $U_{k}:=\varphi_{g_k}(U)\backslash (\cup_{i=1}^{k-1}\varphi_{g_i}(U))$ 
for $1\leq k\leq n$. 
Let $\Psi$ be a bijective map of $\{1,..,n\}\times \mathbb{N}$ onto 
$\mathbb{N}$. 
Define a map $F$ of $X\times\mathbb{N}$ to $U\times\mathbb{N}$ by 
$F ((x,i)) =(\varphi_{g_k^{-1}} (x),\Psi (k,i))$ if $x\in U_k$ for any $1\leq k\leq n$. 
Then it can be easily checked that $F$ is an injective continuous open map such that 
$(y_1, y_2)\in\mathcal{R}_{G, \varphi}^{\infty}$ if and only if 
$(F (y_1),F (y_2))\in\mathcal{R}_{G, \varphi}^{\infty}$. 
\end{proof}
The following lemma is the analogous result of Lemma 2.5 in \cite{B}. 
\begin{lem}\label{lem:Brown}
Let $\mathcal{R}_{G, \varphi}$ be a Cantor minimal system on $X$ and 
$U$ a clopen set in $X$. Then there exists a homeomorphism $\Phi$ of $U\times \mathbb{N}$ 
onto $X\times \mathbb{N}$ such that 
$\Phi\times \Phi((\mathcal{R}_{G, \varphi}|_{U})^\infty) =\mathcal{R}_{G, \varphi}^\infty$ and 
$\Phi(U\times \{1 \})= U\times \{1 \}\subseteq X\times \{1\}$. 
\end{lem} 
\begin{proof}
We shall construct a homeomorphism $\Phi$ of 
$U\times \mathbb{N}\times\mathbb{N}$ onto $X\times \mathbb{N}\times\mathbb{N}$. 
First, for $(x,i,1)\in U\times\mathbb{N}\times\mathbb{N}$, 
define $\Phi ((x,i,1))=(x,i,1)\in X\times \mathbb{N}\times\mathbb{N}$. 
Let $F$ be an injective continuous open map of $X\times \mathbb{N}$ to $U\times \mathbb{N}$ 
in Lemma \ref{lem:lemma of key lemma}. 
Put $E_1=F((X\times \mathbb{N})\backslash (U\times \mathbb{N}))$. 
Define $\Phi ((x,i,2))= (F^{-1}((x,i)), 1)$ if $(x,i)\in E_1$ and 
$\Phi ((x,i,2))= (x,i,2)$ if $(x,i)\in (U\times\mathbb{N})\backslash E_1$. 
Put $E_2:=F((X\times \mathbb{N})\backslash ((U\times\mathbb{N})\backslash E_1))$. 
In a similar way, define $\Phi (((x,i),3)) =(F^{-1}((x,i)),2)$ if 
$(x,i)\in E_2$ and $\Phi ((x,i),3)= (x,i,3)$ if $(x,i)\in 
(U\times\mathbb{N})\backslash E_2$. 
We construct inductively $\Phi$ as follows: 
Let $E_n:=F((X\times \mathbb{N})\backslash((U\times\mathbb{N})\backslash E_{n-1}))$, and define 
$\Phi ((x,i,n+1))= (F^{-1}((x,i)),n)$ if $(x,i)\in E_n$ and 
$\Phi ((x,i,n+1))= (x,i,n+1)$ if $(x,i)\in (U\times \mathbb{N})\backslash E_n$. 
By a Cantor-Bernstein type argument, we obtain the conclusion. 
\end{proof}
\section{Fundamental group}
Let $\mathcal{R}_{G, \varphi}$ be a uniquely ergodic Cantor minimal system 
with an $\mathcal{R}_{G, \varphi}$-invariant probability measure $\mu$. 
Put 
$$
\mathcal{F}(\mathcal{R}_{G, \varphi}):=
\{\mu\times \delta (U):U\text{ is a clopen set in }X\times [n]\text{ such that } 
\mathcal{R}_{G, \varphi}\underset{OE}{\sim} \mathcal{R}_{G, \varphi}^n|_{U}\}
$$
where $\delta$ is the counting measure on $[n]=\{1,..,n\}$. 
\begin{thm}\label{thm:fundamental group}
Let $\mathcal{R}_{G, \varphi}$ be a uniquely ergodic Cantor minimal system 
with an $\mathcal{R}_{G, \varphi}$-invariant probability measure $\mu$. Then 
$\mathcal{F}(\mathcal{R}_{G, \varphi})$ is a countable multiplicative subgroup of 
$\mathbb{R}_{+}^\times$. 
\end{thm}
\begin{proof}
Put 
$$
\mathfrak{S}(\mathcal{R}_{G, \varphi})
=\{\lambda : \mu\times \delta\circ F=\lambda \mu\times \delta ,F\in
\mathrm{Aut}(\mathcal{R}_{G, \varphi}^\infty) \}. 
$$
It is easy to see that $\mathfrak{S}(\mathcal{R}_{G, \varphi})$ is a multiplicative subgroup 
of $\mathbb{R}_{+}^\times$. We shall show that 
$\mathcal{F}(\mathcal{R}_{G, \varphi})=\mathfrak{S}(\mathcal{R}_{G, \varphi})$. 

Let $\lambda\in \mathfrak{S}(\mathcal{R}_{G, \varphi})$, then there exists 
a homeomorphism $F$ of $X\times\mathbb{N}$ such that 
$\mu\times\delta \circ F = \lambda \mu\times\delta$ and 
$F\times F (\mathcal{R}_{G, \varphi}^\infty)=\mathcal{R}_{G, \varphi}^\infty$. 
We see that $\mathcal{R}_{G, \varphi}|_{X\times\{1\}}$ is orbit equivalent to 
$\mathcal{R}_{G, \varphi}^\infty|_{F(X\times\{1\})}$. 
Note that $\mathcal{R}_{G, \varphi}^\infty|_{X\times\{1\}}$ is orbit equivalent to 
$\mathcal{R}_{G, \varphi}$. 
Since $F(X\times\{1\})$ is compact and $F(X\times\{1\})\subseteq\cup_{i=1}^\infty
\{(x,k)\in X\times\mathbb{N}:k=i \}$, there exist a natural number $n$ and 
a clopen set $U$ in $X\times \{1,..,n\}$ such that $F(X\times\{1\})=U$. 
Therefore $\lambda\in \mathcal{F}(\mathcal{R}_{G, \varphi})$. 

Conversely, let $\lambda\in \mathcal{F}(\mathcal{R}_{G, \varphi})$. 
There exist a natural number $n$, a clopen set $U$ in $X\times \{1,..,n\}$ 
and a homeomorphism $h$ of $X$ onto $U\times\{1,..,n\}$ such that 
$h\times h (\mathcal{R}_{G, \varphi})= \mathcal{R}_{G, \varphi}^n|_{U}$. 
By Lemma \ref{lem:Brown}, there exist a homeomorphism $\Phi$ of $U\times\mathbb{N}$ 
onto $X\times\{1,..,n\}\times \mathbb{N}$. 
Let $\Psi$ be a bijective map of $\{1,..,n\}\times \mathbb{N}$ onto 
$\mathbb{N}$, and define a homeomorphism $F$ of $X\times\mathbb{N}$ 
by $F=(id\times \Psi )\circ\Phi \circ (h\times id)$. 
It can be easily checked $F\in\mathrm{Aut}(\mathcal{R}_{G, \varphi}^\infty)$ 
and $\mu\times \delta (F(X\times\{1\}))= \mu\times \delta (U)=\lambda$. 
Therefore $\lambda\in \mathfrak{S}(\mathcal{R}_{G, \varphi})$. 

It is clear that $\mathcal{F}(\mathcal{R}_{G, \varphi})$ is countable because 
$X$ is the Cantor set. 
\end{proof}
\begin{Def}
Let $\mathcal{R}_{G, \varphi}$ be a uniquely ergodic Cantor minimal system 
with an $\mathcal{R}_{G, \varphi}$-invariant probability measure $\mu$. 
We call 
$\mathcal{F}(\mathcal{R}_{G, \varphi})$ the fundamental group of $\mathcal{R}_{G, \varphi}$, 
which is a multiplicative (countable) subgroup of $\mathbb{R}^{\times}_{+}$. 
\end{Def}
\begin{rem}\label{rem:definition}
(i) We need not assume that $\varphi$ is free. 
Moreover we can define the fundamental group of \'etale equivalence relations on 
Cantor sets by Proposition 2.3 in \cite{GPS2}. 
But the associated $C^*$-algebra 
$C(X)\rtimes_{\varphi} G$ ( $C_{r}^*(\mathcal{R})$ ) may not be simple. 
In this paper, we assume that $\varphi$ is free. 
\ \\
(ii) We do not know $\mu\times \delta (U_1 )=\mu\times \delta (U_2)$ implies that 
$\mathcal{R}_{G, \varphi}^{n}|_{U_1}$ is orbit equivalent to $\mathcal{R}_{G, \varphi}^{n}|_
{U_2}$ in general. It is known that if $\mathcal{R}_{G, \varphi}$ is a Cantor minimal 
$\mathbb{Z}^m$-system, then $\mathcal{R}_{G, \varphi}$ satisfies the 
analogous result of Hopf equivalence theorem. See, for example, 
\cite{LO} (Theorem 3.20) and \cite{M} (Theorem 6.11). \ \\
(iii) Since $M(\mathcal{R}_{G, \varphi})=\{\mu\}$, we may regard $\mathcal{R}_{G, \varphi}$ as 
an orbit equivalence relation $\mathcal{R}_{G, T_{\varphi}}$ arising from an ergodic measure 
preserving action $T_{\varphi}$ on a probability measure space $(X, \mu )$. 
It is clear that 
$\mathcal{F}(\mathcal{R}_{G, \varphi})\subseteq \mathcal{F}(\mathcal{R}_{G, T_{\varphi}})$. 
\end{rem}
For an additive subgroup $E$ of $\mathbb{R}$ containing 1, 
we define the positive inner multiplier group $IM_+(E)$ of $E$  by 
$$
IM_+(E) = \{t \in {\mathbb R}_+^{\times} \ | t \in E, t^{-1}  \in E, \text{ and } 
        tE = E \}. 
$$
By Lemma 3.6 in \cite{NW}, there exists a unital subring $R$ of $\mathbb{R}$ such that 
$IM_+(E)=R^\times_{+}$. 
Hence not all countable subgroups of $\mathbb{R}_{+}^{\times}$ 
arise as  $IM_+(E)$. For example, 
$\{9^n \in \mathbb{R}_{+}^{\times} \ | 
n \in {\mathbb Z} \}$ does not arise as $IM_+(E)$ 
for any additive subgroup $E$ of $\mathbb{R}$ containing 1. 

\begin{pro}\label{pro:subseteq}
Let $\mathcal{R}_{G, \varphi}$ be a uniquely ergodic Cantor minimal system 
with an $\mathcal{R}_{G, \varphi}$-invariant probability measure $\mu$. 
Then 
$$
\mathcal{F}(\mathcal{R}_{G, \varphi})\subseteq IM_{+}(T_{\mu}(D_m(\mathcal{R}_{G, \varphi}))) 
=IM_{+}(\{\int_{X} fd\mu :[f]\in D_m(\mathcal{R}_{G, \varphi})\}). 
$$
\end{pro}
\begin{proof}
Let $\lambda\in\mathcal{F}(\mathcal{R}_{G, \varphi})$. 
For any clopen set $U$ in $X\times\{1,..,n\}$, there exist 
clopen sets $V_1,..,V_n$ in $X$ such that $U=\cup_{k=1}^nV_{k}\times \{k\}$. 
Hence Theorem \ref{thm:fundamental group} implies that 
$\lambda ,\lambda^{-1} \in T_{\mu}(D_m(\mathcal{R}_{G, \varphi}))$. 
By the proof of Theorem \ref{thm:fundamental group}, there exist 
homeomorphisms $F_1$ and $F_2$ of $X\times\mathbb{N}$ such that 
$\mu\times\delta \circ F_1 =\lambda \mu\times\delta$ and 
$\mu\times\delta \circ F_2 =\lambda^{-1}\mu\times\delta$. 
Let $V$ be a clopen set in $X$. 
Then $\mu\times\delta (F_1(V\times \{1\}))= \lambda\mu (V)$ and 
$\mu\times\delta (F_2(V\times \{1\}))= \lambda^{-1}\mu (V)$. 
Therefore we see that 
$\lambda T_{\mu}(D_m(\mathcal{R}_{G, \varphi}))\subseteq T_{\mu}(D_m(\mathcal{R}_{G, \varphi}))$ 
and $\lambda^{-1} T_{\mu}(D_m(\mathcal{R}_{G, \varphi}))\subseteq T_{\mu}(D_m(\mathcal{R}_{G, \varphi}))$ 
because $F_1(V\times \{1\})$ and $F_2(V\times \{1\})$ are compact. 
Consequently $\lambda\in IM_{+}(T_{\mu}(D_m(\mathcal{R}_{G, \varphi})))$. 
\end{proof}
We shall show some examples. 
\begin{ex}\label{ex:elementary}
Let $p$ be a prime number, and let $X=\{0,1,..,p-1\}^{\mathbb{N}}$. 
Define $\varphi$ by the addition of $(1,0,0,...)$ with carry over the right 
and $\varphi ((p-1,p-1,...))=(0,0,...)$. 
Then $\mathcal{R}_{\mathbb{Z},\varphi}$ is a uniquely ergodic Cantor minimal system 
with $T_{\mu}(D_m(\mathcal{R}))=\mathbb{Z}[\frac{1}{p}]$. 
Put $U=\{(x_n)_n\in X:x_1=0\}$. 
It is easy to see that $\mathcal{R}_{\mathbb{Z},\varphi}|_U$ is orbit equivalent to 
$\mathcal{R}_{\mathbb{Z},\varphi}$ and $\mu (U)=\frac{1}{p}$. 
(In particular, the induced homeomorphism $\varphi_{U}$ is topologically conjugate to $\varphi$.) 
Hence 
$$
\{p^n:n\in\mathbb{Z}\}\subseteq \mathcal{F}(\mathcal{R}_{\mathbb{Z},\varphi})
\subseteq IM_{+}(T_{\mu}(D_m(\mathcal{R}_{\mathbb{Z},\varphi})))= \{p^n:n\in\mathbb{Z}\}. 
$$
Therefore $\mathcal{F}(\mathcal{R}_{\mathbb{Z},\varphi})=\{p^n:n\in\mathbb{Z}\}$. 
\end{ex}
\begin{ex}
Let $p$ and $q$ be prime numbers such that $gcd(p,q)=1$, and let 
$X=\{0,..,p-1\}^{\mathbb{N}}\times\{0,..,q-1\}^{\mathbb{N}}$. 
Define $\varphi_1$ (resp. $\varphi_2$) by $\varphi_1 ((x,y))= (\tilde{\varphi_1}(x), y)$ 
(resp. $\varphi_2 ((x,y))= (x,\tilde{\varphi_2}(y))$) for $(x,y)\in \{0,..p-1\}^{\mathbb{N}}\times
\{0,..q-1\}^{\mathbb{N}}$ 
where $\tilde{\varphi_1}$ and $\tilde{\varphi_2}$ are the homeomorphisms in 
Example \ref{ex:elementary}. 
Then $\mathcal{R}_{\mathbb{Z}^2, (\varphi_1,\varphi_2)}$ is a uniquely ergodic Cantor minimal 
system with $T(D_m(\mathcal{R}))=\mathbb{Z}[\frac{1}{p}, \frac{1}{q}]$. 
Put $V=\{(x_n)_n\in \{0,..,p-1\}^\mathbb{N}:x_1=0\}$ and $U=V\times \{0,..,q-1\}^{\mathbb{N}}$. 
Then $\mathcal{R}_{\mathbb{Z}^2, (\varphi_1,\varphi_2)}|U$ 
is equal to $\mathcal{R}_{\mathbb{Z}^2,(\varphi_{1U},\varphi_2)}$ where $\varphi_{1U}$ is the 
induced homeomorphism of $\varphi_1$ on $U$. 
Hence it is easy to see that $\mathcal{R}_{\mathbb{Z}^2, (\varphi_1,\varphi_2)}|U$ is 
orbit equivalent to $\mathcal{R}_{\mathbb{Z}^2, (\varphi_1,\varphi_2)}$. 
Therefore we see that 
$\mathcal{F}(\mathcal{R}_{\mathbb{Z}^2, (\varphi_1,\varphi_2)})=\{p^nq^m:n,m\in\mathbb{Z}\}$.
by a similar argument as in Example \ref{ex:elementary}. 
\end{ex}
\begin{ex}
If $\mathcal{R}_{\mathbb{F}^n, \varphi}$ is a uniquely ergodic Cantor minimal system 
arising from a free action of a non-amenable free group of finite rank, 
then $\mathcal{F}(\mathcal{R}_{\mathbb{F}^n, \varphi})=\{1\}$ by Remark \ref{rem:definition} 
(iii) and \cite{G}. 
Note that we do not know whether there exists such a example. 
\end{ex}
We shall show more interesting examples in the next section. 
Giordano-Matui-Putnam-Skau classification theorem enable us to compute 
many examples. 
\begin{thm}\label{thm:computation}
Let $\mathcal{R}_{G, \varphi}$ be a uniquely ergodic Cantor minimal system 
arising from a free action of a finitely generated abelian group. 
Then 
$$\mathcal{F}(\mathcal{R}_{G, \varphi})=IM_{+}(T_{\mu}(D_m(\mathcal{R})))=IM_{+}(\{\int_{X} fd\mu :[f]\in 
D_m(\mathcal{R}_{G, \varphi})\}). 
$$
\end{thm}
\begin{proof}
By the result in \cite{GMPS2}, $\mathcal{R}_{G, \varphi}$ is orbit equivalent to 
a Cantor minimal $\mathbb{Z}$-system. Hence we may assume that $G=\mathbb{Z}$. 

We have 
$\mathcal{F}(\mathcal{R}_{G, \varphi})\subseteq 
IM_{+}(T_{\mu}(D_m(\mathcal{R}_{G, \varphi})))$  by Proposition \ref{pro:subseteq}. 
Conversely, let $\lambda$ be an element in 
$IM_{+}(T_{\mu}(D_m(\mathcal{R}_{\mathbb{Z},\varphi})))$ with $\lambda\leq 1$. 
Then we see that there exists a clopen set $U$ in $X$ such that 
$\mu (U)=\lambda$ by Lemma 2.5 in \cite{GW}. 
Define an additive homomorphism $\psi$ of $T_{\mu}(D_m(\mathcal{R}_{\mathbb{Z},\varphi}))$ 
by $\psi (t)= \lambda t$. Then $\psi$ is an order isomorphism and $\psi (1)=\lambda$. 
Proposition \ref{pro:inclusion} implies that 
$$(D_m(\mathcal{R}_{\mathbb{Z},\varphi}|_{U}),D_m(\mathcal{R}_{\mathbb{Z},\varphi}|_{U})_{+},
[1_{U}]) \cong
(T_{\mu}(D_m(\mathcal{R}_{\mathbb{Z},\varphi})),T_{\mu}(D_m(\mathcal{R}_{\mathbb{Z},\varphi}))_{+},
\lambda).$$ 
Hence 
$\mathcal{R}_{\mathbb{Z},\varphi}|_{U}$ is orbit equivalent to 
$\mathcal{R}_{\mathbb{Z},\varphi}$ by Theorem 2.2 in \cite{GPS1} 
because 
$\mathcal{R}_{\mathbb{Z},\varphi}|_{U}$ is a Cantor minimal $\mathbb{Z}$-system 
(see Section \ref{sec:minimal system} and \cite{GPS1}). 
Therefore $\lambda\in \mathcal{F}(\mathcal{R}_{G, \varphi})$. 
\end{proof}
\begin{cor}
If $R$ is a countable unital subring of $\mathbb{R}$, 
then there exists a uniquely ergodic Cantor minimal $\mathbb{Z}$-system 
$\mathcal{R}_{\mathbb{Z}, \varphi}$ such that $\mathcal{F}(\mathcal{R}_{\mathbb{Z}, \varphi})=R_{+}^\times$. 
Conversely if $\mathcal{R}_{G, \varphi}$ arises from a uniquely ergodic free action of 
a finitely generated abelian group, then there exists a countable unital subring $R$ of $\mathbb{R}$ 
such that $\mathcal{F}(\mathcal{R}_{G, \varphi})=R_{+}^\times$. 
\end{cor}
\begin{proof}
Let $R$ be a countable unital subring of $\mathbb{R}_{+}^\times$. 
By Corollary 6.3 in \cite{HPS}, there exists a uniquely ergodic Cantor minimal $\mathbb{Z}$-system 
$\mathcal{R}_{\mathbb{Z}, \varphi}$ such that 
$$(D_m(\mathcal{R}_{\mathbb{Z},\varphi}),D_m(\mathcal{R}_{\mathbb{Z},\varphi})_{+}, 
[1_{X}])\cong (R,R_{+},1).$$ 
Therefore we see that $\mathcal{F}(\mathcal{R}_{G, \varphi})=R_{+}^\times$ 
by Theorem \ref{thm:computation}. 

Conversely, let $\mathcal{R}_{G, \varphi}$ be a uniquely ergodic Cantor minimal system 
arising from a free action of a finitely generated abelian group. 
Then we have $\mathcal{F}(\mathcal{R}_{G, \varphi})=IM_{+}(T_{\mu}(D_m(\mathcal{R})))$ 
by Theorem \ref{thm:computation}. 
Therefore Lemma 3.6 in \cite{NW} implies that there exists a countable unital subring $R$ of $\mathbb{R}$ 
such that $\mathcal{F}(\mathcal{R}_{G, \varphi})=R_{+}^\times$. 
\end{proof}
By the theorem above, $\{9^n:n\in\mathbb{Z}\}$ cannot be realized as the 
fundamental group of a uniquely ergodic Cantor minimal system arising from a 
free action of a finitely generated abelian group. 

Let $A$ be a simple $C^*$-algebra with a unique trace $\tau$. 
The fundamental group $\mathcal{F}(A)$ of $A$ is defined by the set 
$$
\{ \tau\otimes Tr(p) \in \mathbb{R}^{\times}_{+}\ | \ 
 p \text{ is a projection in } M_n(A) \text{ such that } pM_n(A)p  \cong A \}
$$
where $Tr$ is the unnormalized trace on $M_n(\mathbb{C})$. 
We refer the reader to \cite{Bla} for basic facts of operator algebras. 
We denote by 
$\tau_*$ the map $K_0(A)\rightarrow \mathbb{R}$ induced by a trace $\tau$ on $A$. 
By Proposition 3.7 in \cite{NW}, we have 
$\mathcal{F}(A)\subseteq IM_{+}(\tau_{*}(K_0(A)))$. 
If $\varphi$ is a uniquely ergodic free minimal action of $G$ on $X$, 
then the associated $C^*$-algebra $C(X)\rtimes_{\varphi}G$ is a simple $C^*$-algebra 
with unique trace. 
We do not know whether there exists a relation between orbit 
equivalence of Cantor minimal $G$-systems and $C^*$-isomorphism of associated $C^*$-algebras 
in general. But we have the following theorem. 
\begin{thm}
Let $\mathcal{R}_{\mathbb{Z}^n, \varphi}$ be a uniquely ergodic Cantor minimal system 
arising from a free action of a finitely generated free abelian group $\mathbb{Z}^n$. 
Then $\mathcal{F}(C(X)\rtimes_{\varphi}\mathbb{Z}^n)\subseteq \mathcal{F}
(\mathcal{R}_{\mathbb{Z}^n, \varphi})$. 
\end{thm}
\begin{proof}
The Gap Labeling Theorem (Theorem D.1 in \cite{MS}, see also \cite{Bell}, \cite{BO} and \cite{KP}) 
implies that 
$$
T_{\mu}(D_m(\mathcal{R}_{\mathbb{Z}^n, \varphi}))=\tau_{*}(K_0(C(X)\rtimes_{\varphi}\mathbb{Z}^n)). 
$$
By Theorem \ref{thm:computation}, 
$\mathcal{F}(\mathcal{R}_{\mathbb{Z}^n, \varphi})=IM_{+}(T(D_m(\mathcal{R}_{\mathbb{Z}^n, \varphi})))$. 
Hence we see that 
$\mathcal{F}(C(X)\rtimes_{\varphi}\mathbb{Z})\subseteq \mathcal{F}(\mathcal{R}_{\mathbb{Z}, 
\varphi})$ because we have 
$\mathcal{F}(C(X)\rtimes_{\varphi}\mathbb{Z}) \subseteq IM_{+}(\tau_{*}(K_0(C(X)\rtimes_{\varphi}
\mathbb{Z}^n)))$ by Proposition 3.7 in \cite{NW}. 
\end{proof}
Note that we do not know whether $C(X)\rtimes_{\varphi} \mathbb{Z}^2$ belongs to classifiable 
classes by $K$-groups. 
We shall show that there exists a uniquely ergodic Cantor minimal $\mathbb{Z}$-system 
$\mathcal{R}_{\mathbb{Z}, \varphi}$ 
such that $\mathcal{F}(C(X)\rtimes_{\varphi}\mathbb{Z})\neq \mathcal{F}
(\mathcal{R}_{\mathbb{Z}, \varphi})$. 
By Corollary 6.3 in \cite{HPS} and \cite{EHS}, 
there exists  a uniquely ergodic Cantor minimal $\mathbb{Z}$-system 
$\mathcal{R}_{\mathbb{Z}, \varphi}$ 
such that 
$$
K_0(C(X)\rtimes_{\varphi}\mathbb{Z})=\{(\frac{b}{9^a},c) \in \mathbb{R} \times \mathbb{Z} \ | 
\ a,b,c\in\mathbb{Z},b\equiv c\; \mathrm{mod}\; 8\},
$$
$$
K_0(C(X)\rtimes_{\varphi}\mathbb{Z})_{+}=\{(\frac{b}{9^a},c)\in K_0(A):\frac{b}{9^a}>0\}\cup \{(0,0)\}
\ \ \text{and} \ \   [1_A]_0=(1,1). 
$$ 
Then we see that 
$\mathcal{F}(\mathcal{R}_{\mathbb{Z}, \varphi})=\{3^n:n\in\mathbb{Z}\}$ and 
$\mathcal{F}(C(X)\rtimes_{\varphi}\mathbb{Z})=\{9^n:n\in\mathbb{Z}\}$. 
See Example 3.12 in \cite{NW2}. 

\section{Examples}\label{sec:examples}
First we shall consider a generalization of Example \ref{ex:elementary}. 
Let $\{n_i\}_{i\in\mathbb{N}}$ be a sequence of natural numbers, each greater than or equal to 2 
and let $X=\Pi_{i\in\mathbb{N}}\{0,..,n_i-1\}$. Define 
a homeomorphism $\varphi$ by the addition of $(1,0,0,...)$ with carry over the right 
and $\varphi ((n_1-1,n_2-1,...))=(0,0,...)$. 
Then $\mathcal{R}_{\mathbb{Z},\varphi}$ is a uniquely ergodic Cantor minimal system. 
This system is called the \textit{odometer system} associated with $\{n_i\}_{i\in\mathbb{N}}$. 
For each prime number $p$, define $\epsilon_{p}$ in 
$\mathbb{Z}_{+}\cup\{\infty\}$ by 
$\sup\{m:p^m\;\mathrm{divides}\;\Pi_{i=1}^kn_{i}\;\mathrm{for}\;\mathrm{some}\;k\in\mathbb{N}\}$. 
Let $S$ be the set of prime numbers and 
$N:=\prod_{p\in S}p^{\epsilon_{p}}$. We call $N$ the \textit{super natural number}. 
It is known that two odometer systems $\varphi_1$ and $\varphi_2$ are topologically 
conjugate if and only if their super natural numbers are equal. 
\begin{pro}
Let $\mathcal{R}_{\mathbb{Z},\varphi}$ be an odometer system associated with 
$\{n_i\}_{i\in\mathbb{N}}$. 
Then 
$$\mathcal{F}(\mathcal{R}_{\mathbb{Z},\varphi})=  
\{p_1^{m_1}\cdot\cdot\cdot p_{n}^{m_n}:n\in\mathbb{N}, \epsilon_{p_i}=\infty ,m_i\in\mathbb{N} \}
$$ 
where 
$\epsilon (p)=\sup\{m:p^m\;\mathrm{divides}\;\Pi_{i=1}^kn_{i}\;\mathrm{for}\;\mathrm{some}\;k\in\mathbb{N}\}$. 
\end{pro}
\begin{proof}
Since $T_{\mu}(D_m(\mathcal{R}_{\mathbb{Z},\varphi}))=\{\frac{n}{q}:n\in\mathbb{Z},\;q=p^m\;\mathrm{where}\;p\;\mathrm{is}\;\mathrm{a}\;\mathrm{prime}\;
\mathrm{number}\;\mathrm{and}\\ 0\leq m\leq \epsilon_{p}, m<\infty \}$, 
Theorem \ref{thm:computation} implies the conclusion. 
\end{proof}
We shall consider fundamental groups of certain Denjoy systems. 
A \textit{Denjoy homeomorphism} $\varphi$ is an homeomorphism of $\mathbb{T}$ such that 
the rotation number $\rho (\varphi )$ is irrational and $\varphi$ is not conjugate 
to a pure rotation. Such a homeomorphism has a unique minimal set $\Sigma$ which is 
homeomorphic to the Cantor set. The restriction of a Denjoy homeomorphism to its unique 
minimal set $\Sigma$ is called the \textit{Denjoy system}. 
It is known that a Denjoy system $\mathcal{R}_{\mathbb{Z},\varphi |_{\Sigma}}$ is a uniquely ergodic 
Cantor minimal system. 
We refer the reader to \cite{Mar}, \cite{MSY} and \cite{Put} for details of Denjoy systems. 
It is known that for a Denjoy homeomorphism $\varphi$ there exists an orientation preserving 
surjective continuous map $h$ of $\mathbb{T}$ to $\mathbb{T}$ such that $h\circ \varphi =R_{\rho (\varphi)} 
\circ h$ where $R_{\rho (\varphi)}$ is the rotation such that 
$R_{\rho (\varphi)}([t])=[t+\rho (\varphi)]$. (We regard $\mathbb{T}$ as 
$\{e^{2\pi it}:t\in \mathbb{R} \}$.) 
With notation as above, put $Q(\varphi)=h(\{z\in \mathbb{Z}:\sharp (h^{-1}(h(z)))\neq 1\})$. 
The set $Q(\varphi )$ is is countable and $R_{\rho (\varphi)}$-invariant. 
Conversely, any countable $R_{\rho (\phi)}$-invariant subset of $\mathbb{T}$ can be realized 
as $Q(\varphi)$ of some Denjoy homeomorphism $\varphi$. 
Markley showed that two Denjoy homeomorphisms $\varphi_1$ and $\varphi_2$ are conjugate 
via an orientation preserving conjugating map 
if and only if $\rho (\varphi_1)=\rho (\varphi_2)$ and 
$Q(\varphi_1 )=R_{\theta} (Q(\varphi_2))$ for some $\theta\in (0,1]$. 
\begin{pro}\label{pro:Denjoy}
Let $\theta$ be an irrational number in $[0,1]$ and $\varphi_{\theta}$ a Denjoy homeomorphism with 
$\rho (\varphi_{\theta})=\theta$ and $Q(\varphi_{\theta})=\{e^{2\pi in\theta }:n\in\mathbb{Z}\}$. 
Then $\mathcal{F}(\mathcal{R}_{\mathbb{Z},\varphi_{\theta} |_{\Sigma}})$ is a singly 
generated group. 
More precisely, if $\theta$ is not a 
quadratic number, then 
$\mathcal{F}(\mathcal{R}_{\mathbb{Z},\varphi_{\theta} |_{\Sigma}})=\{1\}$.   
If $\theta$ is a quadratic number, 
then $\mathcal{F}(\mathcal{R}_{\mathbb{Z},\varphi_{\theta} |_{\Sigma}})
= \{\epsilon _0^n  \in R^\times_+ \ | \ n \in \mathbb{Z} \}$ 
and the generator $\epsilon _0$ is given by 
$\epsilon _0 =\frac{t+u\sqrt{D_\theta}}{2}$, 
where 
$D_\theta$ is the  discriminant of $\theta$ and $t,u$ are the positive 
integers satisfying one of the Pell 
equations $t^2-D_\theta u^2=\pm 4$ and $\frac{t+u\sqrt{D_\theta}}{2}$ exceeds $1$ 
and is least. 
\end{pro}
\begin{proof}
By theorem \ref{thm:computation}, we see that 
$\mathcal{F}(\mathcal{R}_{\mathbb{Z},\varphi_{\theta} |_{\Sigma}})=IM_{+} 
(T_{\mu}(D_m(\mathcal{R}_{\mathbb{Z},\varphi_{\theta} |_{\Sigma}})))$. 
Since we have $T_{\mu}(D_m(\mathcal{R}_{\mathbb{Z},\varphi_{\theta} |_{\Sigma}}))=\mathbb{Z}+\mathbb{Z}\theta$ 
(see, for example, \cite{Put}), the same argument as in \cite{NW}(Corollary 3.18) implies 
the conclusion. 
\end{proof}
We shall show some examples. 
\begin{ex}
Let $\theta =\frac{-1+\sqrt{5}}{2}$. Then 
$\mathcal{F}(\mathcal{R}_{\mathbb{Z},\varphi_{\theta} |_{\Sigma}})=\{(\frac{1+\sqrt{5}}{2})^n:n 
\in \mathbb{Z}\}$. 
\end{ex}
\begin{ex}
Let $\theta =\sqrt{5}-2$. Then 
$\mathcal{F}(\mathcal{R}_{\mathbb{Z},\varphi_{\theta} |_{\Sigma}})=\{(\sqrt{5}-2)^n:n 
\in \mathbb{Z}\}$. 
\end{ex}
\begin{ex}
Let $\theta =\frac{1}{\sqrt{5}}$. Then 
$\mathcal{F}(\mathcal{R}_{\mathbb{Z},\varphi_{\theta} |_{\Sigma}})=\{(\sqrt{5}-2)^n:n 
\in \mathbb{Z}\}$. 
\end{ex}

We shall consider a natural generalization to $\mathbb{Z}^2$-actions of Denjoy systems. 
Let $\theta_1$ and $\theta_2$ be irrational numbers and $C$ an countable invariant subset 
of $\mathbb{T}$ under the rotations $R_{\theta_1}$ and $R_{\theta_2}$. 
We can construct a Cantor set $X$ by disconnecting the circle $\mathbb{T}$ along 
any point of $C$. Moreover there exist commutating homeomorphisms $\varphi_1$ and 
$\varphi_2$ of $X$, which are extensions of $R_{\theta_1}$ and $R_{\theta_2}$ respectively. 
See chapter 3 in \cite{CFS}, \cite{FH}, Example 4.4 in \cite{GPS3} and \cite{Put} for details 
and precise definition. 
It is known that this Cantor $\mathbb{Z}^2$-system is a uniquely ergodic 
Cantor minimal system. Moreover the invariant probability measure of this Cantor minimal 
system comes from the normalized Haar measure on $\mathbb{T}$. 
We call this Cantor system the \textit{Denjoy $\mathbb{Z}^2$-system} associated with 
$(\theta_1, \theta_2, C)$. 
We need algebraic number theory of cubic fields to compute some examples of 
fundamental groups of Denjoy $\mathbb{Z}^2$-systems. 
We refer the reader to \cite{SW} for algebraic number theory of cubic fields. 
\begin{ex}
Let $\mathcal{R}_{\mathbb{Z}^2,\varphi}$ be a Denjoy $\mathbb{Z}^2$-system associated with 
$\theta_1=\sqrt[3]{2}$, $\theta_2=\sqrt[3]{4}$ and 
$C=\{e^{2\pi i(n\sqrt[3]{2}+m\sqrt[3]{4})}:n,m\in\mathbb{Z}\}$. 
Then we have $T_{\mu}(D_m(\mathcal{R}_{\mathbb{Z}^2,\varphi}))=
\mathbb{Z}+\mathbb{Z}\sqrt[3]{2}+\mathbb{Z}\sqrt[3]{4}$. 
Since $\mathbb{Z}+\mathbb{Z}\sqrt[3]{2}+\mathbb{Z}\sqrt[3]{4}$ is the ring of 
integers on a cubic field $\mathbb{Q}(\sqrt[3]{2})$, 
$IM_{+}(T_{\mu}(D_m(\mathcal{R}_{\mathbb{Z}^2,\varphi} )))$ is the positive part of 
the unit group of $\mathbb{Q}(\sqrt[3]{2})$, which is equal to 
$\{(1+\sqrt[3]{2}+\sqrt[3]{4})^n:n\in\mathbb{Z}\}$. (See Theorem 13.6.2 in \cite{SW}.) 
Therefore Theorem \ref{thm:computation} implies 
$$\mathcal{F}(\mathcal{R}_{\mathbb{Z}^2, \varphi})=\{(1+\sqrt[3]{2}+\sqrt[3]{4})^n:n\in\mathbb{Z}\}.$$ 
\end{ex}
\begin{ex}
Let $\mathcal{R}_{\mathbb{Z}^2,\varphi}$ be a Denjoy $\mathbb{Z}^2$-system associated with 
$\theta_1=\sqrt[3]{3}$, $\theta_2=\sqrt[3]{9}$ and 
$C=\{e^{2\pi i(n\sqrt[3]{3}+m\sqrt[3]{9})}:n,m\in\mathbb{Z}\}$. 
Then we have $T_{\mu}(D_m(\mathcal{R}_{\mathbb{Z}^2,\varphi}))=
\mathbb{Z}+\mathbb{Z}\sqrt[3]{3}+\mathbb{Z}\sqrt[3]{9}$. 
Since $\mathbb{Z}+\mathbb{Z}\sqrt[3]{3}+\mathbb{Z}\sqrt[3]{9}$ is the ring of 
integers on a cubic field $\mathbb{Q}(\sqrt[3]{3})$, 
$IM_{+}(T_{\mu}(D_m(\mathcal{R}_{\mathbb{Z}^2,\varphi})))$ is the positive part of 
the unit group of $\mathbb{Q}(\sqrt[3]{3})$, which is equal to 
$\{(4+3\sqrt[3]{2}+2\sqrt[3]{4})^n:n\in\mathbb{Z}\}$. 
(See \cite{SW} or Appendix B.3 in \cite{cohen}.) 
Therefore Theorem \ref{thm:computation} implies 
$$\mathcal{F}(\mathcal{R}_{\mathbb{Z}^2, \varphi})=\{(4+3\sqrt[3]{2}+2\sqrt[3]{4})^n:n\in\mathbb{Z}\}.$$ 
\end{ex}
\begin{ex}
Let $\mathcal{R}_{\mathbb{Z}^2,\varphi}$ be a Denjoy $\mathbb{Z}^2$-system associated with 
$\theta_1=2\cos{\frac{2\pi}{7}}$, $\theta_2=4\cos^2{\frac{2\pi}{7}}$ and 
$C=\{e^{2\pi i(n2\cos{\frac{2\pi}{7}}+m4\cos^2{\frac{2\pi}{7}})}:n,m\in\mathbb{Z}\}$. 
Then we have $T_{\mu}(D_m(\mathcal{R}_{\mathbb{Z}^2,\varphi})) 
=\mathbb{Z}+\mathbb{Z}2\cos{\frac{2\pi}{7}}+\mathbb{Z}4\cos^2{\frac{2\pi}{7}}$. 
Note that $2\cos{\frac{2\pi}{7}}$ is a root of the equation $x^3+x^2-2x-1=0$. 
Since $\mathbb{Z}+\mathbb{Z}2\cos{\frac{2\pi}{7}}+\mathbb{Z}4\cos^2{\frac{2\pi}{7}}$ is the 
ring of integers on a cubic field $\mathbb{Q}(2\cos{\frac{2\pi}{7}})$, 
$IM_{+}(T_{\mu}(D_m(\mathcal{R}_{\mathbb{Z},\varphi |_{\Sigma}})))$ is the positive part of 
the unit group of $\mathbb{Q}(2\cos{\frac{2\pi}{7}})$, which is equal to 
$\{(-1+2\cos{\frac{2\pi}{7}}+4\cos^2{\frac{2\pi}{7}})^n(2-4\cos^2{\frac{2\pi}{7}})^m:n,m\in
\mathbb{Z}\}$. 
(See Appendix B.4 in \cite{cohen}.) 
Therefore Theorem \ref{thm:computation} implies 
$$\mathcal{F}(\mathcal{R}_{\mathbb{Z}^2, \varphi})=\{(-1+2\cos{\frac{2\pi}{7}}+
4\cos^2{\frac{2\pi}{7}})^n(2-4\cos^2{\frac{2\pi}{7}})^m:n,m\in\mathbb{Z}\}.$$ 
\end{ex}

\section*{Acknowledgments}
The author would like to thank Professor Matui for informing him about the Gap Labeling 
Theorem. He is also grateful to Professor Watatani for his constant encouragement.

\end{document}